\documentclass[12pt]{amsart}
\usepackage{amsfonts, amssymb, amsmath, amsthm, eucal, latexsym, array, pictex}
\usepackage{fullpage}
\usepackage{verbatim}

\input{cyracc.def}

\newtheorem{thm}{Theorem}
\newtheorem{lm}[thm]{Lemma} 
\newtheorem{cl}[thm]{Corollary}
\newtheorem{prop}[thm]{Proposition}

\theoremstyle{remark}
\newtheorem{rmk}{Remark}

\theoremstyle{definition}

\newcommand{\gt}{\mathfrak}

\newcommand{\ff}{\mathbb F}

\newcommand{\Hom}{\mathrm{Hom}}

\renewcommand{\Im}{{\rm Im\,}}

\newcommand{\ad}{\mathrm{ad\,}}

\newcommand{\g}{\mathfrak g}

\newcommand{\esi}{\varepsilon}

\newcommand{\VV}{{\mathbb V}}

\renewcommand{\le}{\leqslant}
\renewcommand{\ge}{\geqslant}

\newfam\Bbbfam\newfam\eufam\newfam\eusfam%
\font\Bbbfont=msbm10 scaled 1200%
\font\Bbbsmallfont=msbm8%
\textfont\Bbbfam=\Bbbfont\scriptfont\Bbbfam=\Bbbsmallfont
\font\euzw=eufm10 scaled 1200%
\font\euac=eufm7 scaled 1200%
\font\euacc=eufm7 scaled 1000%
\textfont\eufam=\euzw\scriptfont\eufam=\euac%
\scriptscriptfont\eufam=\euacc%
\font\euszw=eusm10 scaled 1200%
\font\eusac=eusm7 scaled 1200%
\font\eusacc=eusm7 scaled 1000%
\textfont\eusfam=\euszw\scriptfont\eusfam=\eusac%
\scriptscriptfont\eusfam=\eusacc%
\newfam\eusfam%
\font\euszw=eusm10 scaled 1200%
\font\eusac=eusm7 scaled 1200%
\font\eusacc=eusm7 scaled 1000%
\textfont\eusfam=\euszw\scriptfont\eusfam=\eusac%
\scriptscriptfont\eusfam=\eusacc%
\begin{document}
\hfill {\scriptsize  March 2, 2010} 
\vskip1ex

\title[derived algebra]
{On the derived algebra of a centraliser}
\author{Oksana Yakimova}
\address{Emmy-Noether-Zentrum, Department Mathematik, 
Universit\"at Erlangen-N\"urnberg\\
Bismarckstrasse 1\,1/2,
91054 Erlangen  Germany}
\email{yakimova@mpim-bonn.mpg.de}
\keywords{Classical Lie algebras, nilpotent elements, centralisers}
\maketitle

\begin{abstract}
Let $\gt g$ be a classical Lie algebra, $e\in\gt g$ a nilpotent 
element and $\gt g_e\subset\gt g$ the centraliser of $e$.
We prove that $\gt g_e=[\gt g_e,\gt g_e]$  if and only if $e$ is rigid. 
It is also shown that if $e\in[\gt g_e,\gt g_e]$, then the nilpotent 
radical of $\gt g_e$ coincides with $[\gt g(1)_e,\gt g_e]$,
where $\gt g(1)_e\subset\gt g_e$ is an eigenspace 
of a characteristic of $e$ with the eigenvalue $1$. 
\end{abstract}

\section*{Introduction}

Let $\gt g$ be a semisimple (or reductive) Lie algebra over an 
algebraically closed field $\ff$ (${\rm char}\,\ff=0$) and $x\in\gt g$.
The main objects of our study here are the centraliser 
$\gt g_x$ and its derived algebra $[\gt g_x,\gt g_x]$.  
There are two natural questions:
for what elements $x$ we have $x\in[\gt g_x,\gt g_x]$ 
and a stronger one, when $\gt g_x=[\gt g_x,\gt g_x]$.
Evidently only a nilpotent element $x$ can satisfies any of these conditions.

Let $e\in\gt g$ be a nilpotent element. 
By the Jacobson-Morozov theorem it can be included 
into an $\gt{sl}_2$-triple
$\{e,h,f\}$ in $\gt g$. 
Set 
$\gt g(\lambda):=\{\xi\in\gt g\mid \ad(h){\cdot}\xi=\lambda\xi\}$ 
and $\gt g(\lambda)_e=\gt g(\lambda)\cap\gt g_e$.
Nilpotent 
elements $e$ such that $e\in[\gt g_e,\gt g_e]$
were studied in \cite{e-gr}, where they are called 
{\it compact}, and in \cite{Dima-reach}, where they are 
called {\it reachable}. 
Extending results of \cite{Dima-reach}, we show that 
if $e$ is reachable and $\gt g$ is a classical Lie algebra, 
then $\gt g(\lambda{+}1)_e=[\gt g(1)_e,\gt g(\lambda)_e]$ for 
all $\lambda$, see Theorem~\ref{g(i)}.

The irreducible components of the algebraic 
varieties $\gt g^{(m)}=\{\xi\in\gt g\mid \dim\gt g_\xi=m\}$ 
are called the {\it sheets} of $\gt g$. 
Their description was obtained by Borho and Kraft in 
\cite{Borho}, \cite{BrKr} in terms of the so-called parabolic 
induction. One of the basic results is that each 
sheet contains a unique nilpotent orbit. 
If a nilpotent orbit coincides with a sheet, it is 
said to be {\it rigid}.
A nilpotent element is said to be {\it rigid}, if its orbit is rigid. 
If $\gt g_e=[\gt g_e,\gt g_e]$, then $e$ is rigid by an almost trivial 
reason (see Proposition~\ref{rigid-1}). 
We prove that in the classical Lie algebras the converse 
is true. This answers  a question 
 put to me by A.\,Premet at the Ascona conference in August (2009). 
In the exceptional Lie algebras there are rigid elements 
such that $\gt g_e\ne [\gt g_e,\gt g_e]$, see Remark~\ref{G2}.
Interest in $[\gt g_e,\gt g_e]$ is motivated by a 
connection with finite $W$-algebras ${\bf U}(\gt g,e)$ 
and their $1$-dimensional 
representations. Recall that ${\bf U}(\gt g,e)$ is a deformations of 
the universal enveloping algebras ${\bf U}(\gt g_e)$ and 
the commutators $[\xi,\eta]$ of $\xi,\eta\in\gt g_e$ naturally 
appear in  the commutator relation of ${\bf U}(\gt g,e)$, 
see e.g. \cite[Section~3.4]{Sasha}. 
As was explained to me by A.\,Premet, the
equality $\gt g_e=[\gt g_e,\gt g_e]$ implies that 
${\bf U}(\gt g,e)$ has at most one non-trivial 
$1$-dimensional representation.

\section{Basis of a centraliser}\label{basis}

In this section
we fix a basis of a centraliser, 
which is used throughout  the paper, and state 
a few easy useful facts. 
Let $\VV$ be a  finite dimensional vector space over
$\ff$ and let $e$ be a nilpotent element of
$\hat{\g}=\gt{gl}(\VV)$. Let $k$ be the number of Jordan blocks of $e$
and $W\subseteq \VV$ a ($k$-dimensional) complement of $\Im e$ in
$\VV$. Let $d_i$ denote the size of the $i$th Jordan block of
$e$. We always assume that the Jordan blocks are ordered such that
$d_1\ge d_2\ge\ldots\ge d_k$ so 
that $e$ is represented by the partition $(d_1,\ldots,d_k)$
of $\dim\VV$. 
Choose a basis $w_1, w_2, \ldots, w_k$ 
in $W$ such that the vectors $e^{j}{\cdot}w_i$ with 
$1\le i\le k$, $0\le j\le d_i{-}1$ form a basis of $\VV$, and put
$\VV[i]:=\textrm{span}\{e^j{\cdot}w_i\,|\,\, j\ge 0\}$. Note that
$e^{d_i}{\cdot}w_i=0$ for all $i$. 

If $\xi\in\hat{\g}_e$, then $\xi(e^j{\cdot}w_i)=e^j{\cdot} \xi(w_i)$, hence
$\xi$ is completely determined by its values on $W$. 
The only restriction on $\xi(w_i)$ is that 
$e^{d_i}{\cdot}\xi(w_i)=\xi(e^{d_i}{\cdot}w_i)=0$. 
Since vectors $e^s{\cdot}w_i$ form a basis of $\VV$, the
centraliser $\hat{\gt g}_e$ has a basis 
$\{\xi_i^{j,s}\}$ such that
$$
\left\{
\begin{array}{l}
\xi_i^{j,s}(w_i)=e^s{\cdot}w_j, \\
\xi_i^{j,s}(w_t)=0 \enskip \mbox{for } t\ne i, \\
\end{array}\right.
\quad 1\le i,j\le k, \ \mbox{ and }\ \max\{d_j-d_i, 0\} \le s\le d_j{-}1 \ .
$$
It is convenient to assume that $\xi_i^{j,s}=0$ whenever 
$s$ does not satisfy the above restrictions. 
An example of $\xi_i^{j,1}$ with $i>j$ and $d_j=d_i{+}1$ is shown in 
Figure~1.

\begin{figure}[htb]
\setlength{\unitlength}{0.023in}
\begin{center}
\begin{picture}(90,85)(-9,-5)

\put(-9,6){$e\!\!:$}
\put(-2.7,5){\vector(0,1){10}}

\put(10,0){\line(1,0){20}}
\put(10,0){\line(0,1){70}}
\put(10,70){\line(1,0){20}}
\put(30,0){\line(0,1){70}}
\put(10,30){\line(1,0){20}}
\put(10,20){\line(1,0){20}}
\put(10,10){\line(1,0){20}}
\put(10,60){\line(1,0){20}}
\put(70,0){\line(0,1){60}}
\put(70,0){\line(1,0){20}}
\put(70,60){\line(1,0){20}}
\put(90,0){\line(0,1){60}}
\put(70,10){\line(1,0){20}}
\put(70,20){\line(1,0){20}}
\put(70,50){\line(1,0){20}}

\qbezier[20](20,32),(20,45),(20,58)
\qbezier[16](80,22),(80,35),(80,48)
\qbezier[15](50,23),(50,40),(50,57)

\put(70,05){\vector(-4,1){40}}                        
\put(70,15){\vector(-4,1){40}}  
\put(70,55){\vector(-4,1){40}}  

 \put(48,65){$\xi_i^{j,1}$}

{\small                                        
\put(11,63){$e^{d_j}{\cdot}w_j$}
\put(12.1,23){$e^2{\cdot}w_j$}
\put(13,13){$e{\cdot}w_{j}$}
\put(15.5,3){$w_{j}$}
\put(17.3,-5.5){$j$} 

\put(72,53){$e^{d_i}{\cdot}w_i$}
\put(74,13){$e{\cdot}w_{i}$}
\put(76.5,3){$w_{i}$}
\put(79,-5.5){$i$}  }
               

\end{picture}
\end{center}
\caption{}\label{pikcha_A}
\end{figure}

The composition rule shows that the basis elements 
$\xi_i^{j,s}$ satisfy the following commutator relation: \\
\begin{equation}\label{commutator}
[\xi_i^{j,s},\xi_p^{q,t}]=\delta_{q,i}\xi_p^{j,t+s}-\delta_{j,p}\xi_i^{q,s+t},
\end{equation}
where $\delta_{i,j}=1$ if $i=j$ and is zero otherwise. 

An $\gt{sl}_2$-triple $\{e,h,f\}$ 
can be chosen in such a way that 
$h{\cdot}w_i=(1-d_i)w_i$.
Then 
\begin{equation}\label{ad-h}
[h,\xi_i^{j,s}]=(d_i-d_j)+2s.
\end{equation} 
Using this equality, 
it is not difficult to describe $h$-eigenspaces 
$\gt g(\lambda)$ in terms of $\xi_i^{j,s}$.
For example,
$\gt g(1)_e$ is generated by 
$\xi_i^{j,0}$ with $d_j=d_i-1$ and $\xi_i^{j,1}$ with $d_j=d_i+1$. 

\vskip1ex

Let $(\ \,,\ )_{{_\VV}}$ be a non-degenerate symmetric or
skew-symmetric bilinear form on $\VV$, i.e., 
$(v,w)_{{_\VV}}=\esi(w,v)_{{_\VV}}$, where $v,w\in\VV$ and $\esi=+1$ or $-1$.
Let $\sigma:\hat{\g}\to\hat{\g}$ be a linear mapping such that 
$(x{\cdot}v,w)_{{_\VV}} =-(v, \sigma(x){\cdot}w)_{{_\VV}}$
for all $v,w\in\VV$ and $x\in\hat{\g}$. Then 
$\sigma$ in an involutive automorphism of $\hat{\g}$.
Let
$\hat{\g}=\,\hat{\g}_0\oplus\gt m$ be the
symmetric decomposition of $\hat{\g}$ 
corresponding to the $\sigma$-eigenvalues $\pm 1$. 
The elements $x\in\gt m$ have the property
that $(x{\cdot}v,w)_{{_\VV}} =(v, x{\cdot}w)_{{_\VV}}$ for all $v,w\in \VV$.

Set $\g:=\hat{\g}_0$ and let $e$ be a nilpotent element of
$\g$. Since $\sigma(e)=e$, the centraliser $\hat{\g}_e$ of
$e$ in $\hat{\g}$ is $\sigma$-stable and
$(\hat{\g}_e)_0=\,\hat{\g}_e^{\sigma}=\,\g_e$. This
yields the $\gt g_e$-invariant symmetric decomposition
$\hat{\gt g}_e=\gt g_e\oplus\gt m_e$.

\begin{lm}   \label{restr}
In the above setting, suppose that $e\in\hat{\gt g}_0$ is a nilpotent element.
Then the cyclic vectors $\{w_i\}$ and thereby the spaces $\{\VV[i]\}$ can be chosen
such that there is an involution $i\mapsto i'$ on the set
$\{1,\dots, k\}$ satisfying the following conditions:  
\begin{itemize}
\item $d_i=d_{i'}$; 
\item $(\VV[i], \VV[j])_{{_\VV}}=0$ if $i\ne j'$;
\item $i=i'$ if and only if $(-1)^{d_i}\esi=-1$.
\end{itemize}
\end{lm}
\begin{proof} 
This is  a standard property of the nilpotent orbits in $\gt{sp}(\VV)$ and
$\gt{so}(\VV)$, see, for example, \cite[Sect.~5.1]{cm} or \cite[Sect.~1]{ja}.
\end{proof}

\subsection{Basis in the orthogonal and symplectic cases.}\label{basis-2}
Let $\{w_i\}$ be a set of cyclic vectors 
chosen according to Lemma~\ref{restr}. Consider the restriction of 
the $\gt g$-invariant form $(\ \,,\ )_{{_\VV}}$  to $\VV[i]+\VV[i']$. 
Since $(w,e^s{\cdot}v)_{{_\VV}}=(-1)^s(e^s{\cdot}w,v)_{{_\VV}}$, 
a vector $e^{d_i{-}1}{\cdot}w_i$ is orthogonal to all 
vectors $e^s{\cdot}w_{i'}$ with $s>0$. Therefore 
$(w_{i'},e^{d_i{-}1}{\cdot}w_i)_{{_\VV}}=(-1)^{d_i{-}1}(e^{d_i{-}1}{\cdot}w_{i'},w_i)_{{_\VV}}\ne 0$.
There is a (unique up to a scalar)  vector $v\in\VV[i]$ such that 
$(v,e^s{\cdot}w_{i'})_{{_\VV}}=0$ for all $s<d_i{-}1$. It is not contained  in 
$\Im e$, otherwise it would be orthogonal to $e^{d_i{-}1}{\cdot}w_{i'}$ too and hence 
to $\VV[i']$. Therefore there is no harm in replacing $w_i$ by $v$. 
Let us always choose the cyclic vectors $w_i$ in such a way 
that $(w_i,e^s{\cdot}w_{i'})_{{_\VV}}=0$ for $s<d_i{-}1$ and 
normalise  them according to: 
\begin{equation}\label{normalise}
(w_i,e^{d_i{-}1}{\cdot}w_{i'})_{{_\VV}}=\pm 1 \ \text{ and } \
(w_i,e^{d_i{-}1}{\cdot}w_{i'})_{{_\VV}}>0 \ \text{ if } \ i\le i'. 
\end{equation}
Then $\gt g_e$ is generated (as a vector
space) by the vectors
$\xi_i^{j,d_j-s}+\varepsilon(i,j,s)\xi_{j'}^{i',d_i-s}$,
where $\varepsilon(i,j,s)=\pm 1$ depending on $i,j$ and $s$ in 
the following way
$$
(e^{d_j-s}{\cdot}w_j,e^{s{-}1}{\cdot}w_{j'})_{{_\VV}}=-\esi(i,j,s)(w_i,e^{d_i{-}1}{\cdot}w_{i'})_{{_\VV}}.
$$
Elements $\xi_i^{j,d_j-s}-\varepsilon(i,j,s)\xi_{j'}^{i',d_i-s}$ form a basis 
of $\gt m_e$. In the following we always 
normalise $w_i$ as above and 
enumerate the Jordan blocks such that 
$i'\in\{i,i+1,i-1\}$ keeping inequalities $d_i\ge d_j$ for $i<j$.
In this basis $\{e^s{\cdot}w_i\}$
the matrix of the restriction of the Killing form to 
$\VV[i]+\VV[i']$ is anti-diagonal with entries $\pm 1$. 

\section{Reachable nilpotent elements}

Let $\gt g$ be a reductive Lie algebra and 
$e\in\gt g$ a nilpotent element. 
We include it into an
$\gt{sl}_2$-triple $\{ e,h,f\}$
and let $\gt g(\lambda)$ stand for the 
$\ad(h)$-eigenspace with eigenvalue $\lambda$.
Set $\gt g(\lambda)_x:=\gt g_x\cap\gt g(\lambda)$.
We fix  a non-degenerate invariant 
bilinear form $\kappa$ on $\gt g$.

In \cite{Dima-reach} an element $x\in\gt g$ 
is called reachable if 
$x\in[\gt g_x,\gt g_x]$. Clearly each 
reachable element is nilpotent. 
Since $\gt g(0)_e$ is a reductive subalgebra, 
the representation of  $\gt g(0)_e$ on $\gt g(1)_e$
is completely reducible. Therefore
$e\in[\gt g_e,\gt g_e]$ if and only if 
$e\in[\gt g(1)_e,\gt g(1)_e]$.
Following \cite{Dima-reach},
we continue to study the derived algebra 
of a centraliser $\gt g_e$ for reachable $e$.

\begin{lm}\label{f}
Let $\hat f(\xi,\eta)=\kappa(f,[\xi,\eta])$ be a 
skew-symmetric form on $\gt g(1)$. Then 
$\hat f$ is non-degenerate on $\gt g(1)_e$. 
\end{lm}
\begin{proof}
If $\gt a_1,\gt a_2\subset\gt g$ are two irreducible
representations of any subalgebra $\gt{sl}_2\subset\gt g$ and 
$\dim\gt a_1\ne\dim\gt a_2$, then 
necessarily $\kappa(\gt a_1,\gt a_2)=0$. 
Applying this to the $\gt{sl}_2$-triple $\{e,h,f\}$ we get
that  
$\kappa$
defines a non-degenerate pairing 
between $\gt g(-1)_f$ and $\gt g(1)_e$. 
It remains to notice that
$\kappa(f,[\gt g(1)_e,\gt g(1)_e])=\kappa([f,\gt g(1)_e],\gt g(1)_e)=
\kappa(\gt g(-1)_f,\gt g(1)_e)$. 
\end{proof}

\begin{rmk}
In the following we need only the fact that $\hat f$ is non-zero 
on $\gt g(1)_e$. But a proof of the weaker statement is not any easier. 
\end{rmk}

Suppose that $e\in\hat{\gt g}=\gt{gl}(\VV)$ 
is given by a partition $((d+1)^m,d^n)$ with both 
$m$ and $n$ being non-zero. 
Then $\hat{\gt g}(0)_e=\gt{gl}_m\oplus\gt{gl}_n$ by \cite[Sect.~3]{ja}.
One can easily compute that 
$$
\hat{\gt g}(1)_e \cong \ff^m\otimes(\ff^n)^*\oplus(\ff^m)^*\otimes\ff^n \enskip
\text{ and } \ \  \ 
\hat{\gt g}(2)_e\cong \gt{gl}_m\oplus\gt{gl}_n
$$
as $\hat{\gt g}(0)_e$-modules.
Let $e=e_{d+1}+e_{d}$ be a decomposition of $e$ 
according to the size of Jordan blocks, 
i.e., $e_d$ is given by the rectangular partition 
$(d^n)$ and $e_{d+1}$ by $((d+1)^m)$. 
Here $e_d{\cdot}w_i=0$, if $w_i$ generates a Jordan 
block of size $d+1$, and $e_{d+1}{\cdot}w_j=0$, if $w_j$ 
generates a Jordan block of size $d$. 
As a representation of $\gt g(0)_e$
the subspace $\hat{\gt g}(2)_e$ decomposes as 
$\gt{sl}_m\oplus\gt{sl}_l\oplus\ff e_d \oplus \ff e_{d+1}$, where 
$[\hat{\gt g}(0)_e,e_d]=[\hat{\gt g}(0)_e,e_{d+1}]=0$.

\begin{lm}\label{2-blocks}
Keep the above assumptions and notation. 
Then 
$[\hat{\gt g}(1)_e,\hat{\gt g}(1)_e]=\gt{sl}_m\oplus\gt{sl}_n\oplus\ff(me_1-ne_2)$
as a $\hat{\gt g}(0)_e$-module.
\end{lm}
\begin{proof}
First, we show that the ``$\gt{sl}$-parts'' of $\hat{\gt g}(2)_e$ 
are contained in $[\hat{\gt g}_e,\hat{\gt g}_e]$. 
Suppose that $m>1$, otherwise $\gt{sl}_m$ is zero. 
According to (\ref{ad-h}), $\xi_{m+1}^{2,1},\xi_1^{m+1,0}\in\hat{\gt g}(1)_e$.
Hence
$[\xi_{m+1}^{2,1},\xi_1^{m+1,0}]=\xi_1^{2,1}\in 
[\hat{\gt g}(1)_e,\hat{\gt g}(1)_e]$.
By the ``zero-trace'' reason, $\xi_1^{2,1}\in\gt{sl}_m$ 
for the irreducible 
$\hat{\gt g}(0)_e$-subrepresentation 
$\gt{sl}_m\subset\hat{\gt g}(2)_e$. Since 
$[\hat{\gt g}(1)_e,\hat{\gt g}(1)_e]$ is $\hat{\gt g}(0)_e$-invariant, 
the whole subspace $\gt{sl}_m$ is contained in it. 
In order to prove the inclusion for the 
``$\gt{sl}_n$-part'' (in case $n>1$), we take
$[\xi_1^{m+1,0},\xi_{m+2}^{1,1}]=\xi_{m+2}^{m+1,1}$.
  
By Lemma~\ref{f}, the skew-symmetric form
$\hat f$ is non-degenerate on $\hat{\gt g}(1)_e$. 
The subspace $\gt{sl}_m\oplus\gt{sl}_n\subset\hat{\gt g}(2)_e$,
being a non-trivial $\hat{\g}(0)_e$-module, 
is orthogonal to $f$ (with respect to $\kappa$), 
hence $[\hat{\gt g}(1)_e,\hat{\gt g}(1)_e]$
contains at least one non-zero vector of the form $ae_1+be_2$. 
Recall that 
$\hat{\gt g}(1)_e=\ff^m{\otimes}(\ff^n)^*\oplus(\ff^m)^*{\otimes}\ff^n$ 
as a representation of $\hat{\gt g}(0)_e$.  
Since  
$[\hat{\gt g}(1)_e,\hat{\gt g}(1)_e]\subset \Lambda^2(\hat{\gt g}(1)_e)$
and $\dim\Lambda^2(\hat{\gt g}(1)_e)^{\hat{\g}(0)_e}\le 1$,
the subspace of $\hat{\gt g}(0)_e$-invariant vectors 
in $[\hat{\gt g}(1)_e,\hat{\gt g}(1)_e]$ 
is at most one dimensional. 
In other words, the subspace  $[\hat{\gt g}(1)_e,\hat{\gt g}(1)_e]$
contains at most one element commuting with $\hat{\gt g}(0)_e$.
Taking $y=[\xi_{m+1}^{1,1},\xi_1^{m+1,0}]$ we get an element 
$y=\xi_1^{1,1}-\xi_{m+1}^{m+1,1}$ in $[\hat{\gt g}(1)_e,\hat{\gt g}(1)_e]$
and $(mn)y-(me_d-ne_{d+1})\in\gt{sl}_m\oplus\gt{sl}_n$. 
\end{proof}

\begin{cl} If $d=1$, then $e_d=0$ and 
$\hat{\g}(2)_e=[\hat{\g}(1)_e,\hat{\g}(1)_e]$.
Otherwise $[\hat{\g}(1)_e,\hat{\g}(1)_e]$
is of codimension $1$ in $\hat{\g}(2)_e$.
\end{cl}

In the following we are going to deal with the
orthogonal and symplectic Lie algebras and freely use results 
and assumptions of subsection~\ref{basis-2}.

\begin{lm}\label{1-block}
Suppose  $\gt g$ is either 
$\gt{sp}(\VV)$ or $\gt{so}(\VV)$ and 
$e\in\gt g$ is given by a rectangular partition $d^k$.
\begin{itemize}
\item \ If $(-1)^d\esi=1$, then $\gt g(0)_e=\gt{sp}_k$ 
and $\gt g(2m)=S^2\ff^k$ for even $m<d$,
$\gt g(2m)=\Lambda^2\ff^k$ for odd $m<d$.
\item \ If $(-1)^d\esi=-1$, then $\gt g(0)_e=\gt{so}_k$ 
and $\gt g(2m)=\Lambda^2\ff^k$ for even $m<d$,
$\gt g(2m)=S^2\ff^k$ for odd $m<d$.
\end{itemize}
\end{lm}
\begin{proof}
The assertions concerning $\gt g(0)_e$ 
follow, for example, from 
\cite[Sect.~3, Prop.~2]{ja}.
Recall that $W\subset\VV$ is a $k$-dimensional complement of $\Im e$.
and each $\xi\in\gt g_e$ is completely determined by its values on 
$W$. Therefore 
$\gt g(2m)_e$ can be identified with the set of 
$\xi\in\Hom(W,e^{m}{\cdot}W)$ such that 
$(\xi(w),e^{d-m-1}w')_{{_\VV}}=-(w,e^{d-m-1}{\cdot}\xi(w'))_{{_\VV}}$. 
Identifying $W$ and $e^{m}{\cdot}W$ by means of $e^{m}$, we see that 
$(\xi(w),w')=-(w,\xi(w')$ for a non-degenerate symmetric or skew-symmetric 
bilinear form $(w,w'):=(w,e^{d-1}{\cdot}w')_{{_\VV}}$ on $W$, and that 
is the only condition on $\xi$. In case the form is 
symmetric, we get $\gt  g(2m)_e\cong \Lambda^2\ff^k$,
and if it is skew-symmetric, then $\gt g(2m)_e\cong S^2\ff^k$.
\end{proof}

\begin{lm}\label{2-block-2}
Let $\gt g$ be either 
$\gt{so}(\VV)$ or $\gt{sp}(\VV)$ and $e\in\gt g$ 
a nilpotent element defined by a partition 
$((d+1)^m,d^n)$. Let $e_d,e_{d+1}$ be as in Lemma~\ref{2-blocks}.
Then 
$me_d-ne_{d+1}\in [\gt g(1)_e,\gt g(1)_e]$. 
\end{lm}
\begin{proof}
By Lemma~\ref{1-block}, 
the subspace $\gt g(2)_e$ decomposes 
into a direct sum of irreducible 
$\gt g(0)_e$-representations as follows
$\gt g(2)_e=\gt a_1\oplus\gt a_2\oplus\ff e_d\oplus\ff e_{d+1}$, 
where $e_d,e_{d+1}$ are the same central vectors as in the
$\gt{gl}(\VV)$ case and $\gt a_1=\gt g\cap\gt{sl}_n$,
$\gt a_2=\gt g\cap\gt{sl}_m$. 
The exact description of subspaces $\gt a_1$ and $\gt a_2$ 
depends on $\gt g$ and the parity of $d$, see Lemma~\ref{1-block}.
In any case 
they both contain no non-zero $\gt g(0)_e$-invariant vectors. 
Using Lemma~\ref{2-blocks}, we get an inclusion 
$[\gt g(1)_e,\gt g(1)_e]\subset\gt a_1\oplus\gt a_2\oplus
 \ff(me_d-ne_{d+1})$. 
The first two subspaces,
$\gt a_1,\gt a_2$, are orthogonal to $f$, but the whole 
commutator is not, due to Lemma~\ref{f}.
Hence $me_d-ne_{d+1}\in[\gt g(1)_e,\gt g(1)_e]$.  
\end{proof}

It is natural to suggest that
the following conditions are equivalent:
\begin{itemize}
\item[({\sf i})] \ $e$ is reachable and $\gt g(0)_e$ is semisimple;  
\item[({\sf ii})] \ $\gt g_e=[\gt g_e,\gt g_e]$.
\end{itemize}
Clearly, condition ({\sf i}) is necessary for ({\sf ii}). 
Below we prove that in the classical Lie algebras 
it is also sufficient.

Given $x\in\hat{\gt g}$ and a non-negative integer
$q$ let $x^q\in\hat{\gt g}$ be the $q$th  power of $x$
(as a matrix). Note that if $\gt g$ 
is either $\gt{so}(\VV)$ or $\gt{sp}(\VV)$, the number $q$ is odd, 
and $x\in\gt g$, then also $x^q\in\gt g$. This remains true for a product 
$x_1\ldots x_q$ of $q$ elements $x_i\in\gt g$. 

\begin{thm}\label{equi-cl}
Suppose that $\gt g$ is a simple classical Lie 
algebra and $e\in\gt g$ a nilpotent element.
Then conditions ({\sf i}) and ({\sf ii}) are equivalent. 
\end{thm}
\begin{proof}
Actually,  we need to show only that ({\sf i}) implies  ({\sf ii}).  

Suppose first that $\gt g=\gt{sl}(\VV)$. 
Then $\gt g(0)_e$ is semisimple exactly
in two cases: $e=0$ and $e$ is regular. In the first 
of them there is nothing to prove. In the second case, where 
$\gt g_e$ is commutative, $e$ is not reachable.  

Suppose $\gt g$ is 
either $\gt{so}(\VV)$ or $\gt{sp}(\VV)$ and 
$e\in\gt g$ satisfies ({\sf i}).  According to 
the description of reachable nilpotent elements 
\cite[Theorem~2.1.(4)]{Dima-reach}, $e$ has Jordan 
blocks  of sizes $(d,d-1,\ldots,1)$
with positive multiplicities $(r_d,\ldots,r_1)$.
Let $e=e_d+e_{d-1}+\ldots+e_1$ be a decomposition 
of $e$ according to the sizes of Jordan blocks.
We assume that $e_i{\cdot}\VV[t]=0$
if $\dim\VV[t]\ne i$. 
Then $e_1=0$ and $e_ie_j=0$ for $i\ne j$. 
Since   $\gt g(0)_e$ is semisimple,  it is contained in  
$[\gt g_e,\gt g_e]$.
Using Lemma~\ref{1-block}, it is not difficult to see that
the nilpotent radical of $\gt g_e$ contains 
$\gt g(0)_e$-subrepresentations of the form 
$\ff^{r_i}\otimes\ff^{r_j}$, 
$\Lambda^2\ff^r$, and $S^2\ff^r$ ($r=r_i,r_j$) of algebras
$\gt{sp}_{r_i}$, $\gt{so}_{r_j}$. 
Each non-trivial representation 
appears also in $[\gt g(0)_e,\gt g_e]$.  
Trivial representations, or elements commuting 
with $\gt g(0)_e$, are either 
$\gt{so}_r$-invariant vectors in $S^2\ff^r$ 
(correspondingly, $\gt{sp}_r$-invariant vectors in $\Lambda^2\ff^r$)
or  $\ff{\otimes}\ff$. 
Vectors of the first type are  $e_i^{2s+1}$,
vectors of the second type
come from pairs $r_i=r_j=1$ with $i\ne j$
as $\xi_p^{q,s}\pm\xi_q^{p,s'}$, where 
$p$th Jordan block is the unique block of size $i$ and 
$q$th Jordan block is the unique block of size $j$. 
  
Recall that $e_1=0$.
Lemma~\ref{2-block-2} implies that 
$e_2\in [\gt g(1)_e,\gt g(1)_e]$. Applying the same 
lemma to the partition $(3^{r_3},2^{r_2})$, we obtain 
that $r_2e_3-r_3e_2$ is contained in $[\gt g(1)_e,\gt g(1)_e]$. 
Hence $e_3\in [\gt g_e,\gt g_e]$. Continuing through all 
sizes $3,4,\ldots,d$  of Jordan blocks we prove that all $e_i$ are elements of
the derived algebra $[\gt g_e,\gt g_e]$.    

Consider the matrix product $e^{2s}e_i\in\gt g$. 
Since $e=e_d+e_{d-1}+\ldots+e_1$ and 
$e_ie_j=0$ for $i\ne j$, we obtain that 
$e^{2s}e_i=e_i^{2s+1}$. Because $e$ is a central 
element in $\gt g_e$,  all its powers commute with 
$\gt g_e$. 
Thus $e^{2s}[\xi,\eta]=[e^{2s}\xi,\eta]$ for 
all $\xi,\eta\in\gt g_e$. Moreover 
$e^{2s}\xi\in\gt g_e$. 
Since $e_i\in [\gt g_e,\gt g_e]$,
we have 
$$ 
e_i^{2s+1}=e^{2s}e_i\in e^{2s}[\gt g_e,\gt g_e] =
 [e^{2s}\gt g_e,\gt g_e]\subset [\gt g_e,\gt g_e]
$$
for all $s$.  

It remains to deal with pairs $(i,j)$, where $i\ne j$ and 
$r_i=r_j=1$. Let $t\mapsto t'$ be the same involution on the set 
of Jordan blocks
as in Lemma~\ref{restr}.
Suppose that the Jordan block of size 
$i$ has number $p$ and the Jordan block of size $j$ 
has number $q$.  Then $p'=p$ and $q'=q$. In particular, 
$i$ and $j$ have the same parity.  Assume that $i>j$. 
Then trivial $\gt g(0)_e$-representations 
associated with the pair $(i,j)$
are generated by the
vectors 
$$
x(s):=\xi_p^{q,s}+(-1)^{s+1}\xi_q^{p,i-j+s} \enskip  \text{ with } \ \ 0\le s\le j-1.
$$
Note that $x(s+1)=[e_i,x(s)]$. 
Thus we only need to show that $x(0)$ is contained in the 
derived subalgebra. 

The Jordan block number $p+1$ has size $i-1$.
Without any doubt, $i-1$ has different from $i$ parity.
Hence $(p+1)'= p+2$ and $p+2<q$.  
Take two elements $y,z\in\gt g_e$:  
$$
y=\xi_p^{p+1,0}-\xi_{p+2}^{p,1}, \enskip
z=\xi_{p+1}^{q,0}-\xi_{q}^{p+2,i-1-j}
$$
and compute their commutator according to~(\ref{commutator}):
$$
[z,y]=\xi_{p+1}^{q,0}\xi_p^{p+1,0} -z\xi_{p+2}^{p,1}
-y\xi_{p+1}^{q,0}-\xi_{p+2}^{p,1}\xi_{q}^{p+2,i-1-j}=
 \xi_p^{q,0}-\xi_q^{p,i-j}=x(0).
$$
This completes the proof.
\end{proof}

In \cite{Dima-reach} a question was raised whether 
the properties $e\in[\gt g_e,\gt g_e]$ and 
$\gt g(\lambda+1)=[\gt g(1)_e,\gt g(\lambda)_e]$ are equivalent. 
The positive answer was given
for $\gt g=\gt{sl}(\VV)$ and $\lambda\ge 1$,  see \cite[Theorem~4.5]{Dima-reach}. 
Here 
we prove the equivalence  for $\gt{sp}(\VV)$ and $\gt{so}(\VV)$.
Of course, if $\gt g(2)_e=[\gt g(1)_e,\gt g(1)_e]$, 
then $e$ is reachable.

\begin{lm}\label{g1}
For any reductive Lie algebra $\gt g$ and any nilpotent 
element $e\in\gt g$ we have $\gt g(1)_e=[\gt g(0)_e,\gt g(1)_e]$.
\end{lm} 
\begin{proof}
Let $\gt t\subset\gt g(0)_e$ be a maximal torus. 
Then 
$\gt z_{\gt g}({\gt t})=\gt t{\oplus}\gt h$, 
where 
$\gt h$ is a reductive subalgebra and $e\in\gt h$. 
Moreover $\gt h_e=\gt h\cap\gt g_e$ contains no 
semisimple elements. In other words, 
$e$ is a {\it  distinguished} nilpotent element in $\gt h$.
Therefore $e\in\gt h$ is even cf. \cite[Theorem~8.2.3]{cm}
and $\gt h(1)_e=0$. It follows that $\gt g(1)_e$ 
contains no non-zero $\gt g(0)_e$-invariant vectors and 
$\gt g(1)_e=[\gt g(0)_e,\gt g(1)_e]$.
\end{proof}

\begin{lm}\label{g2} 
Suppose $\gt g$ is either $\gt{sp}(\VV)$ or 
$\gt{so}(\VV)$ and $e\in\gt g$ is reachable. 
Then $[\gt g(1)_e,\gt g(1)_e]=\gt g(2)_e$. 
\end{lm}
\begin{proof}
Since $e$ is reachable, it has 
Jordan blocks of sizes $(d,d-1,\ldots,1)$ (with positive 
multiplicities). 
Recall that 
$\gt g$ is a symmetric subalgebra of $\hat{\gt g}=\gt{gl}(\VV)$,
see Section~\ref{basis} for more details. 
In other words,  $\hat{\gt g}=\gt g\oplus\gt m$ is 
a $\mathbb Z_2$-grading and $\hat{\gt g}_e=\gt g_e\oplus\gt m_e$.
Suppose $\xi_i^{j,s},\xi_a^{b,t}$ 
are non-commuting elements of $\hat{\gt g}(1)_e$. 
Commutator relation~(\ref{commutator}) implies that  either $j=a$ or $i=b$. 
Without loss of generality 
we may (and will) assume that $i=b$. 
Let $d_i$ be the size of the $i$th Jordan block and 
$d_j$ the size of the $j$th Jordan block. 
By (\ref{ad-h}), $d_i$ and $d_j$
have different parity and the same holds for Jordan blocks with numbers 
$i$ and $a$. 
Take an element $x\in\gt g(1)_e$ such that 
$x=\xi_i^{j,s}+\esi(i,j,s)\xi_{j'}^{i',s'}$.
Because of the parity conditions $j'\ne i$ and $i'\ne a$.
Therefore $[x,\xi_a^{i,t}]=[\xi_i^{j,s},\xi_a^{i,t}]$.
It follows that 
$$
[\gt g(1)_e,\hat{\gt g}(1)_e]=[\hat{\gt g}(1)_e,\hat{\gt g}(1)_e]=
  \hat{\gt g}(2)_2=\gt g(2)_e\oplus\gt m(2)_e.
$$  
To conclude the proof note that 
$\hat{\gt g}(1)_e=\gt g(1)_e\oplus\gt m(1)_e$ and 
$[\gt m(1)_e,\gt g(1)_e]\subset\gt m(2)_e$. 
Hence $[\gt g(1)_e,\gt g(1)_e]=\gt g(2)_e$.
\end{proof}

\begin{lm}\label{g(i)} 
Suppose that $\gt g$ is either $\gt{sp}(\VV)$ or 
$\gt{so}(\VV)$ and $e\in\gt g$ is  reachable. 
Then $[\gt g(1)_e,\gt g(\lambda)_e]=\gt g(\lambda+1)_e$ for all 
$\lambda\ge 0$. 
\end{lm}
\begin{proof}
If $\lambda$ is odd, the proof of Lemma~\ref{g2} goes 
practically without changes.
Let $\xi_i^{j,s},\xi_a^{b,t}$ 
be non-commuting elements of $\hat{\gt g}(1)_e$ and 
$\gt g(\lambda)_e$, without specifying which vector lies 
in what subspace. Assuming $\lambda$ is odd, one can say that 
the sizes of the $i$th  and $j$th Jordan blocks are of different parity and 
the sizes of the $a$th and $b$th Jordan blocks are also of different parity. 
Since $\xi_i^{j,s}$ and $\xi_a^{b,t}$  do not commute, using (\ref{commutator}) we get that 
either $j=a$ or $i=b$
and 
one still may assume that $i=b$. Proceeding as in the proof of
Lemma~\ref{g2}, we get
$$
[\gt g(1)_e,\hat{\gt g}(\lambda)_e]+[\hat{\gt g}(1)_e,\gt g(\lambda)_e]=
\hat{\gt g}(\lambda)_e.
$$
Now notice that 
$\hat{\gt g}(\lambda)_e=\gt g(\lambda)_e\oplus\gt m(\lambda)_e$
for all $\lambda$
and $[\gt g,\gt m]\subset\gt m$.

Suppose now that $\lambda$ is even and 
$\hat x=\xi_i^{j,s}\in\hat{\gt g}(1)_e$, 
$y=\xi_a^{b,t}\in\hat{\gt g}(\lambda)_e$ are non-commuting 
elements. According to (\ref{ad-h}), sizes of Jordan blocks with numbers 
$a$ and $b$ are of the same parity. In particular, 
if $i\in\{a,b\}$, then $j,j'\not\in\{a,b,a',b'\}$ and 
if $j\in\{a,b\}$, then $i,i'\not\in\{a,b,a',b'\}$. 
Take again $x=\hat x+\esi(i,j,s)\xi_{j'}^{i',s'}$.
Then $[\hat x,y]=[x,y]$ apart from two 
cases $\{i,i'\}=\{a,b\}$ and $\{j,j'\}=\{a,b\}$,
where $[\hat x,y]$ is either 
$\xi_{i'}^{j,s+t}$ or $-\xi_i^{j',s+t}$. 
Let $\gt a\subset\hat{\gt g}_e$ be a subspace generated 
by all $\xi_i^{j,t}$ such that $|d_i-d_j|=1$. We have shown that 
$\gt g(\lambda+1)_e\subset [\gt g(1)_e,\gt g(\lambda)_e]+\gt a+\gt m_e$. 
Note that $\gt a=(\gt a\cap\gt g)\oplus(\gt a\cap\gt m)$. 
Hence it remains to prove that 
$\gt a\cap\gt g(\lambda+1)_e\subset [\gt g(1)_e,\gt g(\lambda)_e]$.

Suppose $|d_i-d_j|=1$ and 
let $x(s)=\xi_i^{j,d_j-s}+\esi(i,j,s)\xi_{j'}^{i',s'}$ be 
an element of $\gt g_e$. Set $D:=\min(d_i,d_j)$. 
Then $x(D)\in\gt g(1)_e$ and, by (\ref{ad-h}), 
$[h,x(s)]=(1+2(D-s))x(s)$. 
Our goal is to show that each $x(s)$ 
with $2(D-s)=\lambda$ lies in $[\gt g(1)_e,\gt g(\lambda)_e]$.
This will complete the proof.  

Since the sizes of the $i$th and $j$th Jordan blocks are of different parity,
either $i'=i$ or $j'=j$. Without loss of generality, 
we may (and will) assume that $i=i'$. Then necessary $j\ne j'$. 
Assume first that $t=D-s$ is odd. Then $\xi_i^{i,t}\in\gt g_e$.
Using~(\ref{commutator}), it is straightforward to compute 
that $[x(D),\xi_i^{i,t}]=x(s)$. Here $\xi_i^{i,t}\in\gt g(\lambda)_e$
as required. Assume now that $t$ is even. Then 
$y=\xi_j^{j,t}-\xi_{j'}^{j',t}\in\gt g(\lambda)_e$ and again 
$x(s)=[y,x(D)]$.  
\end{proof}

\begin{rmk}
It is also possible to verify 
the equality $\gt g(\lambda+1)_e=[\gt g(\lambda)_e,\gt g(1)_e]$
directly, by writing down bases of $\gt g(\lambda)_e$, $\gt g(1)_e$ and 
computing the commutators. 
\end{rmk}

\section{Rigid nilpotent elements}

The irreducible components of the quasi-affine varieties 
$\gt g^{(m)}=\{\xi\in\gt g\mid \dim\gt g_\xi=m\}$ are called the 
sheets of $\gt g$. Their description was obtained by 
Borho and Kraft \cite{Borho}, \cite{BrKr}. One of the main 
results is that each sheet contains exactly  one nilpotent orbit. 
Nilpotent orbits, which coincide with sheets, are said to be rigid, 
and all their elements are said to be rigid as well. 
In the classical Lie algebras the classification of rigid nilpotent
elements was obtained by Kempken \cite[Subsection~3.3]{gisela}. 
(Note that rigid orbits are called {\it original} in  \cite{gisela}.) 
Namely, an element $e$ 
of $\gt{so}(\VV)$ or $\gt{sp}(\VV)$ is rigid 
if and only if it is given by a partition  
$(d^{r_d},(d-1)^{r_{d-1}},\ldots,1^{r_1})$ with 
all multiplicities $r_i$ being positive  and  
\begin{itemize}
\item \ if $\gt g=\gt{so}(\VV)$ and $d_i$ is odd, then $r_i\ne 2$;
\item \ if $\gt g=\gt{sp}(\VV)$ and $d_i$ is even, then $r_i\ne 2$.
\end{itemize}
In view of \cite[Sect.~3]{ja}, the last two conditions mean that 
$\gt g(0)_e$ has no factors isomorphic to $\gt{so}_2$ and therefore 
is semisimple.

\begin{prop}\label{rigid-1} 
Let $\gt g$ be a semisimple Lie algebra and 
$x\in\gt g$ such that $\gt g_x=[\gt g_x,\gt g_x]$.
Then $x$ is a rigid nilpotent element.
\end{prop}
\begin{proof}
Assume that $e$ is not rigid. Then any sheet $S$,
containing $e$, contains also non-nilpotent elements and 
they form a non-empty open subset in $S$. In particular, there 
is a curve $x:\ff\to S$ such that 
$\dim\gt g_{x(t)}=\dim\gt g_e$ for all $t\in\ff$, $\lim_{t\to 0}x(t)=e$, and
$x(t)$ is not nilpotent for all $t\ne 0$.
Thereby 
$[\gt g_{x(t)},\gt g_{x(t)}]\ne \gt g_{x(t)}$ for all 
non-zero values of $t$. In the limit, dimension of 
the commutant cannot increase. 
Hence $\dim[\gt g_e,\gt g_e]<\dim\gt g_e$. This contradiction 
completes the proof.  
\end{proof}

\begin{thm}\label{rigid-cl}
Let $\gt g$ be a simpe classical Lie 
algebra and $e\in\gt g$ a nilpotent element.
Then $\gt g_e=[\gt g_e,\gt g_e]$ 
if and only if $e$ is rigid. 
\end{thm}
\begin{proof}
Due to Proposition~\ref{rigid-1}, we need to prove 
only one implication. Assume that $e$ is rigid. 
In type $A$ this means that $e=0$ and there is nothing 
to prove. 

Suppose that $\gt g$ is either $\gt{sp}(\VV)$ or $\gt{so}(\VV)$. 
Comparing descriptions of Kempken \cite{gisela}, reproduced above,
and of Panyushev \cite[Theorem~2.1.(4)]{Dima-reach}, we conclude that $e$ is reachable.
It follows from the Kempken's results, that $\gt g(0)_e$ is semisimple. 
Therefore we have: $e$ is reachable and $\gt g(0)_e$ is semisimple.
This is exactly condition $({\sf i})$, and the equality
$\gt g_e=[\gt g_e,\gt g_e]$ holds by  
Theorem~\ref{equi-cl}.
\end{proof}

\begin{rmk}\label{G2}
In exceptional Lie algebras there are rigid nilpotent 
elements such that $e\not\in[\gt g_e,\gt g_e]$. 
The simplest example is provided by a short root 
vector $e$ in the simple Lie algebra of type $G_2$.
Since $\dim(\gt g_e)=6$,  this element is rigid.
Here 
$\gt g(1)_e=0$ and therefore $e$ is not reachable.  
\end{rmk}

\vskip0.3ex

\noindent
{\bf Acknowledgements.}  
 Final part of this work  was carried out during 
my stay at the Max-Planck-Institut f\"ur Mathematik (Bonn). 
I am grateful to this institution for warm hospitality and support.

\end{document}